\documentclass[preprint]{article}

\usepackage{lineno,hyperref}
\usepackage{graphicx}
\usepackage{subfig}
\usepackage{amsmath,amsthm} 
\usepackage{epsfig}
\usepackage{threeparttable}
\usepackage{bm}
\usepackage{balance}
\usepackage{mathtools}
\usepackage{epstopdf}
\usepackage{arxiv}

\usepackage{xcolor}
\usepackage{ulem}

\usepackage{amssymb}

\usepackage{algorithm}
\usepackage{algorithmic}

\theoremstyle{plain}
\newtheorem{theorem}{Theorem}
\newtheorem{lemma}{Lemma}

\newtheorem{definition}{Definition}
\newtheorem{claim}{Claim}

\usepackage{array}
\newcolumntype{P}[1]{>{\centering\arraybackslash}p{#1}}

\modulolinenumbers[5]

\usepackage{natbib}
\title{Extreme points of general transportation polytopes}

\author{ Patrice Koehl \\
	Department of Computer Science\\
	University of California, Davis\
	Davis, CA 95616 \\
	\texttt{koehl@cs.ucdavis.edu}
}

\date{}

\begin{document}
\maketitle

\begin{abstract}
Transportation matrices are $m\times n$ non-negative matrices whose row sums
and row columns are equal to, or dominated above with given integral vectors $R$ and $C$.
Those matrices belong to a convex polytope whose extreme points have been previously characterized.
In this article, a more general set of non-negative transportation matrices is considered,
whose row sums are bounded by two integral non-negative vectors $R_{min}$ and $R_{max}$ and column sums
are bounded by two integral non-negative vectors $C_{min}$ and $C_{max}$.
It is shown that this set is also a convex polytope whose extreme points are then fully characterized.
\end{abstract}



\section{Introduction}
\label{sec:intro}

Imagine that there are $n$ agents that are assigned to perform $m$ tasks, and assume that each agent can perform an integer number of tasks and each task can only be assigned to an integer number of agents.
A manager in charge of the assignment agent / task will consider the ``cost" of assigning each agent to each task and will
design an ``optimal assignment", namely an assignment that leads to a minimal overall cost.
Finding a solution to this seemingly simple practical problem has led to the development of a gem in the mathematics and statistics communities, namely the optimal transport (OT) problem.

A mathematical formulation of the problem is to consider two sets of points $S_{1}$ of size $n$ (the agents) and $S_{2}$ of size $m$ (the tasks).
Each point $k$ in $S_1$ (resp $S_2$) is assigned a ``mass" $m_{1}(k)$, the amount of tasks it can perform (resp $m_{2}(k)$, the number of agents needed to perform this task). 
The cost of an assignment between $S_1$ and $S_2$ is encoded as a non-negative matrix $T=[T(i,j)]$ with $i\in\{1\ldots,n\}$ and $j\in\{1\ldots,m\}$.
The OT problem can then be
formulated as finding a matrix $G$ of correspondence between points in $S_1$ and points in $S_2$ that minimizes the total transport cost
 $U$ defined as 
 \begin{equation*}
U(G)= \sum_{i=1}^{n} \sum_{j=1}^{m} T(i,j) G(i,j).
\label{eqn:energy} 
\end{equation*} 
The minimum of $U$ is to be found for the values of $G(i,j)$ that satisfy the following constraints
\begin{subequations}
\label{eqn:constraintsOT}
\begin{align}
& \forall (i,j), \quad G(i,j) \ge 0, \label{eqn:constraintsOT1} \\
 &  \forall i, \ \ \quad \sum\limits_j G(i,j) =m_1(i), \label{eqn:constraintsOT2}\\
 &  \forall j, \ \ \quad \sum\limits_i G(i,j) =m_2(j).  \label{eqn:constraintsOT3}
 \end{align}
  \end{subequations}
What makes the OT problem so interesting is that its solution includes two essential components.
First, it defines a distance between the distributions considered ($m_1$ and $m_2$). 
These distances have enabled statisticians and mathematicians to derive a geometric structure
on the space of probability distributions (\cite{Villani:2008, Peyre:2018}).
Second, it also provides the optimal transportation plan $G$ between the distributions; this optimal plan defines a
registration, thereby enabling alignment between the distributions.
Applications of OT have exploded in the recent years, in domains such as applied mathematics, machine learning, computer vision, 
and linguistics (see \cite{Villani:2008, Cotar:2013, Santambrogio:2015, Peyre:2018}, among others).

The formulation defined above corresponds to the Monge-Kantorovich version of the optimal transport problem (\cite{Kantorovich:1942, Villani:2008}). 
A simplified version of the OT problem is to set the masses to 1 (i.e. each agent can only perform one task and each task is performed by one agent) and to assume
balance (same number of agents and tasks). The transportation matrix $A$ then becomes integral (its values are integers) and satisfies the following constraints:
\begin{subequations}
\label{eqn:constraintsA}
\begin{align}
& \forall (i,j), \quad A(i,j) \in \{0,1\}, \label{eqn:constraintsA1} \\
 &  \forall i, \ \ \quad \sum\limits_j A(i,j) = 1, \label{eqn:constraintsA2}\\
 &  \forall j, \ \ \quad \sum\limits_i A(i,j) = 1.  \label{eqn:constraintsA3}
 \end{align}
  \end{subequations}
This simplified problem that has its root in Monge's original transport problem \cite{Monge:1781} is itself a classical problem in combinatorial optimization referred to as the assignment problem or alternatively, using the language of graph theory, as the bipartite weighted matching problem (for a comprehensive analysis of assignment problems, see for example \cite{Burkard:2009}).

Our focus in this paper is on the structures of the set of matrices that either satisfy the constraints \ref{eqn:constraintsOT} or \ref{eqn:constraintsA}.
Those sets of matrices belong to so-called transportation polytopes that have been studied extensively in the literature (see for example, \cite{Birkhoff:1946, Brualdi:1976, Brualdi:1980, Muchlis:1992, Gill:2009, Ziegler:2012, Loera:2013, Cao:2019, Chen:2021}).
In the special case $n=m$ and $m_1(i)=m_2(j) = 1 \quad \forall (i,j)$, the matrices $G$ are doubly stochastic.
The set of all those matrices form a convex polytope $\Omega_n$. Under the same conditions, the set of matrices $A$ is the set of permutations $P_n$.
The set $P_n$ is the set of extreme points of $\Omega_n$; as such, any matrix $G$ in $\Omega_n$ can be expressed as a linear combination of
permutation matrices in $P_n$:

\begin{theorem} 
\label{th:th1}
(\cite{Birkhoff:1946, Neumann:1953}) An $n\times n$ matrix G is doubly stochastic if and only if there is a finite set of permutations matrices $P_1,\ldots,P_N$ and corresponding non-negative real numbers $\alpha_1, \ldots, \alpha_N$ with $\alpha_1+\ldots+\alpha_N=1$ such that $A=\alpha_1 P_1 + \ldots + \alpha_N P_N$.
\end{theorem}

There are many proofs available for this theorem, referred to as the Birkhoff - von Neumann theorem; we only refer here to the original proofs \cite{Birkhoff:1946, Neumann:1953}. This theorem has proved useful to establish convergence for algorithms developed to solve the assignment problem (see for example \cite{Kosowsky:1994, Koehl:2021b}.

The equations \ref{eqn:constraintsA} refer to a very specific balanced assignment problem.
There are, however many more types of assignment problems, usually referred to as unbalanced problems.
They consider a number of agents that differ from the number of tasks (using the example and terminology from above).
The most common formulation assigns a single agent to one task, leaving some agents and/or some task unmatched.
It is referred to in the literature as the $k$-cardinality assignment problem (\cite{Dell:1997}).
Some other formulations allow for multiple jobs to be assigned to the same agent to compensate
for the imbalance (when the number of tasks is bigger than the number of agents), with possibly additional constraints
such as each agent is allocated at least one task.
Transportation matrices for such problems belong to a more general class of transportation polytopes than the doubly stochastic matrices and doubly substochastic matrices.
In this paper, we propose to characterize those transportation polytopes.

We define these more general assignment problems as follows. The cost of transport between $S_1$ and $S_2$ is still encoded 
as a non-negative matrix $T=[T(i,j)]$ with $i\in\{1\ldots,n\}$ and $j\in\{1\ldots,m\}$.
The general assignment problem is formulated as finding a matrix $G$ of correspondence between points in $S_1$ and points in $S_2$ that minimizes the total transport cost $U = \sum_{i=1}^{n} \sum_{j=1}^{m} T(i,j) G(i,j)$
whose values $G(i,j)$ satisfy the following constraints
\begin{subequations}
\label{eqn:constraintsG}
\begin{align}
& \forall (i,j), \ \ \quad G(i,j) \in \{0,1\}, \label{eqn:constraintsG1} \\
 &  \forall i, \ \ \quad r_i \le \sum\limits_j G(i,j) \le R_i ,\label{eqn:constraintsG2}\\
 &  \forall j, \ \ \quad c_j \le \sum\limits_i G(i,j) \le C_j,  \label{eqn:constraintsG3}
 \end{align}
  \end{subequations}
where $r_i$ and $R_i$ are given positive integers satisfying $0 \le r_i \le R_j$ and similarly $c_j$ and $C_j$ are given positive
integers satisfying $0 \le c_j \le C_j$.
A relaxed version of these constraints is to replace equation \ref{eqn:constraintsG1} with $0 \le G(i,j) \le 1$, i.e. allowing $G(i,j)$ to be fractional.
Finally, a possible additional constraint is to preset the total number of assignments between $S_1$ and $S_2$ to a given integer value $k$, i.e.
$\displaystyle \sum_{i=1}^{n} \sum_{j=1}^{m} G(i,j) = k$.
Possible solutions of the relaxed general assignment problem belong to sets of matrices that are defined below.

\begin{definition}
Let $R_{min} = (r(1),\ldots, r(n))$, $R_{max} = (R(1),\ldots, R(n))$, $C_{min} = (c(1),\ldots, c(m))$, and $C_{max} = (C(1),\ldots, C(m))$ be four
non-negative integral vectors satisfying
\begin{eqnarray*}
& \forall i, \ \ \quad 0 \le r(i) \le R(i),\\
& \forall j, \ \ \quad 0 \le c(j) \le C(j).
\end{eqnarray*}
Let $A$ be a non-negative matrix of size $n\times m$ and let us denote the row sum vector of $A$ as $RA$ and the column sum vector of $A$ as $CA$.
Let us also define $\sigma(A)$ to be the sum of all elements of $A$, i.e. $\sigma(A) = \sum_{i=1}^{n} \sum_{j=1}^{m} A(i,j)$.
The transportation polytope $\mathcal{U}(R_{min}^{max},C_{min}^{max})$ is the set of $n\times m$ matrices $A$ that satisfy
\begin{eqnarray*}
& \forall (i,j), \ \ \quad 0 \le A(i,j) \le 1,\\
 &  \forall i, \ \ \quad r(i) \le RA(i) \le R(i), \\
 &  \forall j, \ \ \quad c(j) \le CA(j) \le C(j) .
\end{eqnarray*}
The transportation polytope $\mathcal{U}^k(R_{min}^{max},C_{min}^{max})$ is the set of $n\times m$ matrices $A$ that satisfy
\begin{eqnarray*}
\mathcal{U}^k(R_{min}^{max},C_{min}^{max}) = \left \{ A \in \mathcal{U}(R_{min}^{max},C_{min}^{max}) | \quad \sigma(A) = k \right \}.
\end{eqnarray*}
\end{definition}
Denote $\mathcal{P}(R_{min}^{max},C_{min}^{max})$ the set of all matrices in $\mathcal{U}(R_{min}^{max},C_{min}^{max})$ whose entries are either 0 or 1 with a similar definition for $\mathcal{P}^k(R_{min}^{max},C_{min}^{max})$ with respect to $\mathcal{U}^k(R_{min}^{max},C_{min}^{max})$.
We now state our main result.

\begin{theorem} 
\label{th:th2}
The transportation polytopes $\mathcal{U}(R_{min}^{max},C_{min}^{max})$ and $\mathcal{U}^k(R_{min}^{max},C_{min}^{max})$ satisfy the following properties:
\begin{itemize}
\item[a)] $\mathcal{U}(R_{min}^{max},C_{min}^{max})$ is the convex hull of all matrices in $\mathcal{P}(R_{min}^{max},C_{min}^{max})$,
\item[b)] $\mathcal{U}^k(R_{min}^{max},C_{min}^{max})$ is the convex hull of all matrices in $\mathcal{P}^k(R_{min}^{max},C_{min}^{max})$.
\end{itemize}
\end{theorem}

\paragraph{Remarks:}
\begin{itemize}
\item[i)] If $R_{min}=R_{max}=1_n$, $C_{min}=C_{max}=1_m$ (where $1_n$ and $1_m$ are vectors of one of size $n$ and $m$, respectively), $n=m=k$, $\mathcal{U}^k(R_{min}^{max},C_{min}^{max})= \Omega_n$ and $\mathcal{P}^k(R_{min}^{max},C_{min}^{max})= P_n$ and theorem \ref{th:th2} is then equivalent to the Birkhoff-Von Neuman theorem for doubly stochastic matrices, theorem \ref{th:th1}.
\item[ii)] If $R_{min}=0_n$, $C_{min}=0_m$, $R_{max}=1_n$, $C_{max}=1_m$, $\mathcal{U}(R_{min}^{max},C_{min}^{max})$ is the set of doubly substochastic matrices, $\mathcal{P}(R_{min}^{max},C_{min}^{max})$ is the set of subpermutation matrices; a specific version of theorem \ref{th:th2} was established (see for example \cite{Mirsky:1959, Cihak:1970}).
\item[iii)] Similar to case ii), if $R_{min}=0_n$, $C_{min}=0_m$, $R_{max}=1_n$, $C_{max}=1_m$, $\mathcal{U}^k(R_{min}^{max},C_{min}^{max})$ is the set of doubly substochastic matrices with total sum $k$, and $\mathcal{P}(R_{min}^{max},C_{min}^{max})$ is the set of subpermutation matrices of rank $k$; a specific version of theorem \ref{th:th2} was established by Mendelsohn and Dulmage for square matrices (\cite{Mendelsohn:1958}), and later by Brualdi and Lee for rectangular matrices (\cite{Brualdi:1978}).
\end{itemize}

\section{A simple proof of Birkhoff - von Neumann theorem}

As mentioned above, there are many proofs available for the Birkhoff - von Neumann theorem, some of which belong now  to textbooks. 
Here we describe a simple proof.
It is not original, but will serve as the basis for elements of the proof of theorem \ref{th:th2}.

First, we note that it is straightforward to show that $\Omega_n$ is a non-empty compact convex set in $\mathbb{R}^{n\times m}$ 
and that any matrix in $P_n$ is an extreme point of $\Omega_n$. Since any non-empty compact convex set is the convex hull of its extreme points,
to finish the proof, we only need to show that any matrix $A \in \Omega_n\smallsetminus P_n$  is not an extreme point of $\Omega_n$.

We prove first the following claim.
\begin{claim}
Let $A \in \Omega_n$. If a row or column of $A$ contains a fractional value, then it contains at least two.
\label{claim:claim1}
\end{claim}
\begin{proof}
Let $A(i,j)$ be a fractional value of $A$. Since
\begin{eqnarray*}
\sum_k A(i,k) = 1,
\end{eqnarray*}
there exists $j_2 \in [1,n]\smallsetminus \{j\}$ such that $A(i,j_2)$ is fractional. Similarly, since
\begin{eqnarray*}
\sum_k A(k,j) = 1,
\end{eqnarray*}
there exists $i_2 \in [1,n]\smallsetminus \{i\}$ such that $A(i_2,j)$ is fractional.
\end{proof}

Let now $A$ be a matrix in $\Omega_n\smallsetminus P_n$. 
There exists a pair $(i_1,j_1)$ such that $A(i_1,j_1)$ is fractional.
Based on claim \ref{claim:claim1}, there exists $j_2 \in [1,n]$ with $j_2 \ne j_1$ such that $A(i_1,j_2)$ is non integral.
Similarly, we can find $i_2 \in [1,n]$ with $i_2 \ne i_1$ such that $A(i_2,j_2)$ is fractional.
We can continue in this manner, leading to a path $\left( (i_1, j_1), (i_1 j_2), \ldots \right)$ with fractional values in $A$. 
As $n$ and $m$ are finite, we will ultimately reach a pair that we have already visited. This means that we have identified a loop $L$ among all edges between $S_1$ and $S_2$; the cardinality of this cycle is even (bipartite graph). We write this cycle as
\begin{eqnarray*}
L = \{ (a_1, b_1), (a_2, b_2), \ldots, (a_{2M}, b_{2M})\},
\end{eqnarray*}
where $2M = |L|$. 
We define the matrix $N$
\begin{eqnarray}
\begin{cases}
	N(i,j) &= 0 \quad (i,j) \notin C \nonumber \\
	N(a_{2k},b_{2k}) &= 1 \quad k\in \{1,\ldots, M\} \nonumber \\
	N(a_{2k-1},b_{2k-1}) &= -1 \quad k\in \{1,\ldots, M\}.
\end{cases}
\label{eqn:N}
\end{eqnarray}
Let us now define
\begin{eqnarray*}
\epsilon_{max} = \min \{ A(a_1,b_1), \ldots, A(a_{2M}, b_{2M}), 1 - A(a_1,b_1), \ldots,1-A(a_{2M}, b_{2M}) \}.
\end{eqnarray*}
As all elements in the loop $L$ are fractional, $0 < \epsilon_{max} < 1$. For $\epsilon \in (0, \epsilon_{max}]$, we define $E_1 = A + \epsilon N$ and $E_2 = A - \epsilon N$.
As two consecutive pairs in $L$ lead to the addition and subtraction of the same quantity $\epsilon$ on one row or one column of $A$, it is easy to verify 
that $E_1$ and $E_2$ are doubly stochastic and therefore belong to $\Omega_n$.
Since $A = \frac{1}{2}(E_1+E_2)$, $A$ is not an extreme point of $\Omega_n$.

\section{The polytopes $\mathcal{U}(R,C)$ and $\mathcal{U}^k(R,C)$}

Let us start with the simple case for which $R_{min}=R_{max}=R$ and $C_{min}=C_{max}=C$. We rewrite $\mathcal{U}(R_{min}^{max},C_{min}^{max})$ and $\mathcal{U}^k(R_{min}^{max},C_{min}^{max})$ as simply $\mathcal{U}(R,C)$ and $\mathcal{U}^k(R,C)$. Those polytopes have been studied extensively (see for example \cite{Jurkat:1967, Brualdi:1980, Rothblum:1989, Brualdi:2005, Fonseca:2009, Brualdi:2006, Cavenagh:2013, Chen:2016}).
We prove first the following lemma.
\begin{lemma}
The polytope $\mathcal{U}(R,C)$ is non-empty if and only if $\sum_{i=1}^n R(i) = \sum_{j=1}^m C(j)$. In addition,
$\mathcal{U}^k(R,C)$ is non empty if and only if $k = \sum_{i=1}^n R(i)$, in which case $\mathcal{U}^k(R,C) = \mathcal{U}(R,C)$.
\end{lemma}
\begin{proof}
Let $A$ be a matrix in $\mathcal{U}(R,C)$. Then,
\begin{eqnarray*}
\sigma(A) &=& \sum_{i=1}^{n} \sum_{j=1}^{m} A(i,j) = \sum_{i=1}^{n} \left( \sum_{j=1}^{m} A(i,j) \right) = \sum_{i=1}^{n} R(i) \\
&=&  \sum_{j=1}^{m} \left( \sum_{i=1}^{n} A(i,j) \right) = \sum_{j=1}^{m} C(j).
\end{eqnarray*}
Therefore, $\sum_{i=1}^n R(i) = \sum_{j=1}^m C(j)$. Conversely, if $S=\sum_{i=1}^n R(i) = \sum_{j=1}^m C(j)$, it is straightforward
to build a matrix $A$ that belongs to $\mathcal{U}(R,C)$ (for example $A=\frac{1}{S} RC^T$).

Let $A$ be a matrix in $\mathcal{U}^k(R,C)$. By definition, $k=\sigma(A)$. Since $\mathcal{U}^k(R,C) \subset \mathcal{U}(R,C)$, we get $k=\sum_{i=1}^n R(i)$.
In addition, let $A$ be a matrix in $\mathcal{U}(R,C)$. Then $\sigma(A)=\sum_{i=1}^{n} R(i) = k$, therefore $A \in \mathcal{U}^k(R,C)$. 
\end{proof}

In the case considered here, theorem \ref{th:th2} is reduced to the following lemma:
\begin{lemma}
\label{lem:lem1}
The transportation polytopes $\mathcal{U}(R,C)$ and $\mathcal{U}^k(R,C)$ satisfy the following properties:
\begin{itemize}
\item[a)] $\mathcal{U}(R,C)$ is the convex hull of all matrices in $\mathcal{P}(R,C)$,
\item[b)] $\mathcal{U}^k(R,C)$ is the convex hull of all matrices in $\mathcal{P}^k(R,C)$.
\end{itemize}
\end{lemma}
The proof of property $a)$ of lemma \ref{lem:lem1} is similar to the simple proof provided above for the Birkhoff - von Neumann theorem, once we replace claim \ref{claim:claim1} with the following claim:
\begin{claim}
Let $A \in \mathcal{U}(R,C)$. If a row or column of $A$ contains a fractional value, then it contains at least two.
\label{claim:claim2}
\end{claim}
\begin{proof}
Let $A(i,j)$ be a fractional value of $A$. Since
\begin{eqnarray*}
\sum_k A(i,k) = R(i)
\end{eqnarray*}
and $R(i)$ is a non-negative integer, necessarily $R(i) \ne 0$ and there exists $j_2 \in [1,n]\smallsetminus \{j\}$ such that $A(i,j_2)$ is fractional. Similarly, since
\begin{eqnarray*}
\sum_k A(k,j) = C(j)
\end{eqnarray*}
and $C(j)$ is a non-negative integer, necessarily $C(j) \ne 0$ there exists $i_2 \in [1,n]\smallsetminus \{i\}$ such that $A(i_2,j)$ is fractional.
\end{proof}

The proof of part $b)$ requires in addition that the matrices $E_1$ and $E_2$ built from the matrix $A$ in $\mathcal{U}^k(R,C)$ also belong to $\mathcal{U}^k(R,C)$. This is a direct consequence of the fact that the loop identified in the matrix $A$ has en even number of elements. As such  the matrix $N$ defined in equation \ref{eqn:N} satisfies $\sigma(N)=0$, and therefore $\sigma(E_1)=\sigma(E_2)=\sigma(A)=k$, i.e. $E_1$ and $E_2$ belong to $\mathcal{U}^k(R,C)$.

\section{The polytope $\mathcal{U}(R_{min}^{max},C_{min}^{max})$}

We start with the following lemma:
\begin{lemma}
 $\mathcal{U}(R_{min}^{max},C_{min}^{max})$ is convex.
 \label{lem:lem2}
 \end{lemma}
 \begin{proof}
 The proof of lemma \ref{lem:lem2} is relatively straightforward. We provide it here for sake of completeness.
Let $A$ and $B$ be two matrices belonging to $\mathcal{U}(R_{min}^{max},C_{min}^{max})$, $\alpha$ a real number in $[0,1]$, and $E = \alpha A + (1-\alpha) B$. 
We have:
\begin{eqnarray*}
E(i,j) = \alpha A(i,j) + (1-\alpha) B(i,j).
\end{eqnarray*}
As both $A(i,j)$ and $B(i,j)$ belong to $[0,1]$ and $[0,1]$ is convex, $E(i,j) \in [0,1]$. Also, for $i \in [1,n]$,
\begin{eqnarray*}
 \sum_{j=1}^{m} E(i,j) = \alpha  \sum_{j=1}^{m} A(i,j) + (1-\alpha) \sum_{j=1}^{n} B(i,j).
 \end{eqnarray*}
 As $\sum_{j=1}^{m} A(i,j)$ and $\sum_{j=1}^{m} B(i,j)$ belong to $[r(i), R(i)]$ and this interval is convex,  $\sum_{j=1}^{n} E(i,j) \in[r(i), R(i)]$. This is true for all $i \in[1, n]$.
 
 Similarly, 
 \begin{eqnarray*}
 \sum_{i=1}^{n} E(i,j) = \alpha  \sum_{i=1}^{n} A(i,j) + (1-\alpha) \sum_{1=1}^{n} B(i,j).
 \end{eqnarray*}
As $\sum_{i=1}^{n} A(i,j)$ and $\sum_{i=1}^{n} B(i,j)$ belong to $[c(j),C(j)]$ and this interval is convex,  $\sum_{i=1}^{n} E(i,j) \in [c(j),C(j)]$. This is true for all $j \in[1, m]$.
  
 Therefore, $\alpha A + (1-\alpha) B$ belongs to $\mathcal{U}(R_{min}^{max},C_{min}^{max})$ and this set is convex.
 \end{proof}
 
\begin{lemma}
 If $A$ belongs to $\mathcal{P}(R_{min}^{max},C_{min}^{max})$, then $A$ is an extreme point of $\mathcal{U}(R_{min}^{max},C_{min}^{max})$.
 \label{lem:lem3}
 \end{lemma}
 \begin{proof}
 Again, the proof of lemma \ref{lem:lem3} is relatively straightforward. We provide it here for sake of completeness.
 Let $A$ be a matrix of $\mathcal{P}(R_{min}^{max},C_{min}^{max})$ and let us suppose that $A = \frac{1}{2}(E+F)$ where $E$ and $F$ are two \textit{distinct} matrices in $\mathcal{U}(R_{min}^{max},C_{min}^{max})$. 
 
 For any entry $A(i,j)=0$, we have:
 \begin{eqnarray*}
 \frac{1}{2}( E(i,j) + F(i,j)) = 0.
 \end{eqnarray*}
 Since $E(i,j)\ge 0$ and $F(i,j)\ge 0$, we have $E(i,j)=F(i,j)=0$.
 
 For any entry $A(i,j)=1$, we have:
  \begin{eqnarray*}
 \frac{1}{2}( E(i,j) + F(i,j)) = 1,
 \end{eqnarray*}
 i.e.,
   \begin{eqnarray*}
 E(i,j) + F(i,j) = 2.
 \end{eqnarray*}
Since $0\le E(i,j)\le 1$ and $0\le F(i,j)\le 1$, we have $E(i,j)=F(i,j)=1$.

Therefore, $A=E=F$, which is in contradiction with the hypothesis that $E$ and $F$ are distinct. Therefore, $A$ is not the midpoint of a line segment whose endpoints are in 
$\mathcal{U}(R_{min}^{max},C_{min}^{max})$, and so $A$ is an extreme point of $\mathcal{U}(R_{min}^{max},C_{min}^{max})$.
\end{proof}

Finally, we prove the following lemma:
\begin{lemma}
If a matrix $A$ is an extreme point of $\mathcal{U}(R_{min}^{max},C_{min}^{max})$ then it belongs to $\mathcal{P}(R_{min}^{max},C_{min}^{max})$.
 \label{lem:lem4}
 \end{lemma}
 \begin{proof}
 We use a proof by contrapositive, loosely inspired by the proof in \cite{Deng:2019} for doubly substochastic matrices. Let $A$ be a matrix in $\mathcal{U}(R_{min}^{max},C_{min}^{max})$. 
 
 We start with some definitions. An entry in $A$  being neither 0 nor 1 is called a fractional entry, and a row (resp. column) containing at least one fractional entry is called a fractional row (resp. column). 
 A fractional line is either a fractional row or a fractional column. 
 
 As $A \in \mathcal{U}(R_{min}^{max},C_{min}^{max}) \smallsetminus \mathcal{P}(R_{min}^{max},C_{min}^{max})$, 
it has at least one fractional entry and therefore one fractional line. We consider two cases.
\begin{itemize}
\item[Case 1)] \textit{All fractional lines include at least two fractional values.}

The proof is then very similar to the simple proof provided above for the Birkhoff - von Neumann theorem, with no need for claim \ref{claim:claim1}.
This means that we can find a loop $L$ within $A$ whose cardinality  is even (bipartite graph). As before, we write this cycle as
\begin{eqnarray*}
L = \{ (a_1, b_1), (a_2, b_2), \ldots, (a_{2M}, b_{2M})\},
\end{eqnarray*}
where $2M = |L|$. 
We then define the matrix $N$
\begin{eqnarray*}
\begin{cases}
	N(i,j) &= 0 \quad (i,j) \notin C \nonumber \\
	N(a_{2k},b_{2k}) &= 1 \quad k\in \{1,\ldots, M\} \nonumber \\
	N(a_{2k-1},b_{2k-1}) &= -1 \quad k\in \{1,\ldots, M\}.
\end{cases}
\end{eqnarray*}
We also define
\begin{eqnarray*}
\epsilon_{max} = \min \{ A(a_1,b_1), \ldots, A(a_{2M}, b_{2M}), 1 - A(a_1,b_1), \ldots,1-A(a_{2M}, b_{2M}) \}.
\end{eqnarray*}
As all elements in the loop $L$ are fractional, $0 < \epsilon_{max} < 1$. For $\epsilon \in (0, \epsilon_{max}]$, we define $E_1 = A + \epsilon N$ and $E_2 = A - \epsilon N$.
As two consecutive pairs in $L$ leads to the addition and subtraction of the same quantity $\epsilon$ on one row or one column, it is easy to verify 
that the row sums and column sums of $E_1$ and $E_2$ are equal to the row sums and row columns $A$, and therefore $E_1$ and $E_2$ belong to $\mathcal{U}(R_{min}^{max},C_{min}^{max})$.
Since $A = \frac{1}{2}(E_1+E_2)$, $A$ is not an extreme point of $\mathcal{U}(R_{min}^{max},C_{min}^{max})$.

\item[Case 2)] \textit{There exists at least one fractional line that includes a single fractional value.}

Let us assume that one such fractional line is a column $j_1$ (the proof would be the same if it were a row) and let  $A(i_1,j_1)$ be the only fractional value on $j_1$. If row $i_1$ only contains one fractional value, we terminate. Otherwise $i_1$ includes at least two fractional values and therefore we can find $j_2\ne j_1$ such that $A(i_1,j_2)$ is fractional. We continue in this manner, until either (a) we have reached a pair that we have already visited,  or (b) we have reached a line that only contains one fractional value. If we are in case (a), we have identified a loop of fractional values whose cardinality is even and the proof is then similar to \textit{Case 1}. If we are in case (b), we have identified an incomplete loop, i.e. a cycle. If the cycle ends because the following column had a single fractional entry, the cardinality of the cycle is even, and we can again follow exactly the proof from \textit{Case 1}. Otherwise, the cycle has an odd number of elements; we write it as
\begin{eqnarray*}
C = \{ (i_1, j_1), (i_2, j_2), \ldots, (i_{2K-1}, j_{2K-1})\},
\end{eqnarray*}
where $2K-1 = |C|$. We define the matrix $N$ as:
\begin{eqnarray*}
\begin{cases}
			N(i,j) &= 0 \quad (i,j) \notin C \nonumber \\
                        N(i_{2k},j_{2k})  &= 1   \quad k\in \{1,\ldots, K-1\} \nonumber \\
                        N(i_{2k-1},j_{2k-1})  &= -1   \quad k\in \{1,\ldots, K\} .                   
                    \end{cases}
\end{eqnarray*}
Let $\epsilon_{max}=\min\{ A(i_1,j_1), \ldots, A(i_{2K-1},j_{2K-1}), 1-A(i_1,j_1), \ldots, 1-A(i_{2K-1},j_{2K-1})\}$. Since all the values $A(i,j)$ when $(i,j)\in C$ are fractional, $0 < \epsilon_{max} <1$.

Let us define $E_1(\epsilon)=A+\epsilon N$ and $E_2(\epsilon)=A-\epsilon N$. 
If $0 < \epsilon < \epsilon_{max}$, all values in $E_1(\epsilon)$ and $E_2(\epsilon)$ are in the interval $[0,1]$. In addition, by construction, any line sum of $E_1(\epsilon)$ and $E_2(\epsilon)$ is equal to 
the sum of the same line of $A$, \textit{with the exception of column $j_1$ and row $i_{2K-1}$} . Those two lines, however, include a single fractional value. As such, the column sum $CA(j_1)$  and the row sum $RA(i_{2K-1})$ in $A$ are fractional. Since the bounds on those row / column are integer values, 
\begin{eqnarray*}
 r(i_{2K-1})  &< RA(i_{2K-1}) < R(i_{2K-1}), \\
 c(j_1)  &< CA(j_1) < C(j_1),
 \end{eqnarray*}
and it is possible to modify the values of  $A(i_1, j_1)$ and $A(i_{2K-1}, j_{2K-1})$ without compromising the constraints on the corresponding column and row, respectively. Let us define $\epsilon_2 = \min\{CA(j_1)-c(j_1), RA(i_{2K-1})-r(i_{2K-1}), C(j_1)-CA(j_1), R(i_{2K-1})-RA(i_{2K-1})\}$; note the $\epsilon_2$ is strictly fractional, i.e. $0 < \epsilon_2 < 1$.
 If we set  $\epsilon$ such that $0 < \epsilon < \min(\epsilon_{max}, \epsilon_2)$, then $E_1(\epsilon)$ and $E_2(\epsilon)$ belong to $\mathcal{U}(R_{min}^{max},C_{min}^{max})$. Since $A = \frac{1}{2}(E_1(\epsilon)+E_2(\epsilon))$, $A$ is not an extreme point of $\mathcal{U}(R_{min}^{max},C_{min}^{max})$.
 
 \end{itemize}
Note that there are no other cases, as any such case would imply that the matrix $A$ does not contain any fractional line, and therefore would belong to  $\mathcal{P}(R_{min}^{max},C_{min}^{max})$. This concludes the proof of lemma \ref{lem:lem4}.
\end{proof}

\textit{The proof of theorem \ref{th:th2} a)}  is then directly the consequence of lemma \ref{lem:lem2}, \ref{lem:lem3}, and \ref{lem:lem4}.

\section{The polytope $\mathcal{U}^k(R_{min}^{max},C_{min}^{max})$}
The set of matrices $\mathcal{U}^k(R_{min}^{max},C_{min}^{max})$ is the subset of the polytope $\mathcal{U}(R_{min}^{max},C_{min}^{max})$,
whose members $A$ satisfy the additional constraint that the total sum of  their elements, $\sigma(A)$ is set to a given positive integer value $k$. Note that this
set is non empty only if $0 < k \le \sum_{i=1}^n R(i)$ and $k \le \sum_{j=1}^n C(j)$. Here we show that $\mathcal{U}^k(R_{min}^{max},C_{min}^{max})$ forms a convex polytope whose extreme values are the matrices in $\mathcal{P}^k(R_{min}^{max},C_{min}^{max})$ using the same strategy we used for $\mathcal{U}(R_{min}^{max},C_{min}^{max})$.

\begin{lemma}
 $\mathcal{U}^k(R_{min}^{max},C_{min}^{max})$ is convex.
 \label{lem:lem5}
 \end{lemma}
\begin{proof}
The proof  is straight forward. Briefly, let $A$ and $B$ be two matrices belonging to $\mathcal{U}^k(R_{min}^{max},C_{min}^{max})$, $\alpha$ a real number in $[0,1]$, and $E = \alpha A + (1-\alpha) B$.  To prove that $E \in \mathcal{U}^k(R_{min}^{max},C_{min}^{max})$, we need to show that:
\begin{itemize}
\item[1)] $E(i,j) \in [0,1]$: this comes from the convexity of $[0,1]$ (see proof of lemma \ref{lem:lem2}).
\item[2)] $r(i) \le \sum_{j=1}^{m} E(i,j) \le R(i)$ for all $i\in [1,n]$: this comes from the convexity of $[r(i),R(i)]$ (see proof of lemma \ref{lem:lem2}).
\item[3)] $c(j) \le \sum_{i=1}^{n} E(i,j) \le C(j)$ for all $j\in [1,m]$: this comes from the convexity of $[c(j),C(j)]$ (see proof of lemma \ref{lem:lem2}).
\item[4)] $\sigma(E) = \sum_{i=1}^{n} \sum_{j=1}^{m} E(i,j) = k$: this comes from the linearity of the operator $\sigma$.
\end{itemize}
\end{proof}

\begin{lemma}
 If $A$ belongs to $\mathcal{P}^k(R_{min}^{max},C_{min}^{max})$, then $A$ is an extreme point of $\mathcal{U}^k(R_{min}^{max},C_{min}^{max})$.
 \label{lem:lem6}
 \end{lemma}
 \begin{proof}
 Again, the proof of lemma \ref{lem:lem6} is straightforward and identical to the proof of lemma \ref{lem:lem3}.
 \end{proof}
 
 Finally, we prove:
 \begin{lemma}
If a matrix $A$ is an extreme point of $\mathcal{U}^k(R_{min}^{max},C_{min}^{max})$ then it belongs to $\mathcal{P}^k(R_{min}^{max},C_{min}^{max})$.
 \label{lem:lem7}
 \end{lemma}
 \begin{proof}
 We use a proof by contrapositive with similar ideas as those used for the equivalent proof for lemma \ref{lem:lem4}. Let $A \in \mathcal{U}^k(R_{min}^{max},C_{min}^{max}) \smallsetminus \mathcal{P}^k(R_{min}^{max},C_{min}^{max})$. As such, it has at least one fractional entry and therefore one fractional line. We consider two cases.
\begin{itemize}
\item[Case 1)] \textit{All fractional lines include at least two fractional values.}
The proof is then equivalent to the proof of \textit{Case 1} for lemma \ref{lem:lem4}.

\item[ Case 2)] \textit{There exists at least one fractional line that includes a single fractional value.}

Let us assume that one such fractional line is a column $j_1$ (the proof would be the same if it were a row) and let  $A(i_1,j_1)$ be the only fractional value on $j_1$. Just like in \textit{Case 2} of lemma \ref{lem:lem4}, we proceed from this value to generate a cycle of fractional values. If the cardinality of this cycle is even, we are done. A problem arises, however, if the cardinality of the cycle is odd:
\begin{eqnarray*}
C = \{ (i_1, j_1), (i_2, j_2), \ldots, (i_{2K-1}, j_{2K-1})\},
\end{eqnarray*}
where $2K-1 = |C|$. We could still define the matrix $N$ as:
\begin{eqnarray}
\begin{cases}
			N(i,j) &= 0 \quad (i,j) \notin C \nonumber \\
                        N(i_{2k},j_{2k})  &= 1   \quad k\in \{1,\ldots, K-1\} \nonumber \\
                        N(i_{2k-1},j_{2k-1})  &= -1   \quad k\in \{1,\ldots, K\} .                   
                    \end{cases}
\label{eqn:N}
\end{eqnarray}
However, if we define $E_1(\epsilon)=A+\epsilon N$ and $E_2(\epsilon)=A-\epsilon N$, even with an $\epsilon$ small enough, we have $\sigma(E_1(\epsilon))=\sigma(A) -\epsilon$ and $\sigma(E_2(\epsilon))=\sigma(A)+\epsilon$, i.e. $E_1(\epsilon)$ and $E_2(\epsilon)$ do not belong to $\mathcal{U}^k(R_{min}^{max},C_{min}^{max})$.

To address this issue, let us first define that a row $i$ (resp. a column $j$) is \textit{mutable} if and only if $r(i) < RA(i) < R(i)$ (resp.  $c(j) < CA(j) < C(j)$) where $RA(i)$ is the sum of the elements of row $i$ and $CA(j)$ is the sum of the elements of column $j$.
In other word, a row or column is mutable if it is possible to add or remove a well chosen small non-zero value to one of its elements without violating the constraints on its sum.
For example, the column $j_1$ of $A$ defined above is mutable: it contains a single fractional value, hence $CA(j_1)$ is fractional and therefore satisfies the definition of mutable.
We now make the following claim:
\begin{claim}
If the matrix $A$ contains one mutable column, then it contains at least 2 mutable columns.
\label{claim:claim4}
\end{claim}
\begin{proof}
Let $l$ be a mutable column of $A$. 
if all other columns $j$ of $A$ were non mutable, then their column sums $CA(j)$ would all be integer values, and $\sigma(A)=CA(l) + \sum_{j=1, j\ne l}^{m}CA(j)$ would then be fractional, which is a contradiction with the fact that $\sigma(A)=k$. Therefore, there exists a least another column of $A$ that is mutable.
\end{proof}

Based on claim \ref{claim:claim4}, as $j_1$ is mutable, there exists another column $j_2 \ne j_1$ that is mutable. We consider then two cases:
\begin{itemize}
\item [a)] \textit{The cycle $C$ include $j_2$}. 

We truncate the cycle to the first vertex that belongs to the column $j_2$. This truncated cycle, $C'$, includes an even number of vertices.
We can then conclude as in \emph{Case 1} by considering this cycle.

\item[b)] \textit{The cycle $C$ does not include $j_2$}. 

Let  $A(i_2,j_2)$ be a fractional value on $j_2$. We build a second cycle $D$ in $A$ starting from this value. If the cardinality of this cycle is even, we can follow exactly the proof from \textit{Case 1}. Otherwise, the cycle has an odd number of elements; we write it as
\begin{eqnarray*}
D = \{ (i'_1, j'_1), (i'_2, j'_2), \ldots, (i'_{2M-1}, j'_{2M-1})\},
\end{eqnarray*}
where $2M-1 = |D|$. We define the matrix $N'$ as:
\begin{eqnarray*}
\begin{cases}
			N'(i,j) &= 0 \quad (i,j) \notin C \nonumber \\
                        N'(i_{2k},j_{2k})  &= -1   \quad k\in \{1,\ldots, M-1\} \nonumber \\
                        N'(i_{2k-1},j_{2k-1})  &= 1   \quad k\in \{1,\ldots, M\} .                   
                    \end{cases}
\end{eqnarray*}
For $\epsilon$ small enough (i.e. smaller than the characteristic $\epsilon$ for the two cycles),  we define $E_1(\epsilon)=A+\epsilon N + \epsilon N'$ and $E_2(\epsilon)=A-\epsilon N - \epsilon N'$ where the matrix $N$ was previously defined in equation \ref{eqn:N}. 
With the right choice of $\epsilon$, All values in $E_1(\epsilon)$ and $E_2(\epsilon)$ are in the interval $[0,1]$. In addition, by construction, any row/column sum of $E_1(\epsilon)$ and $E_1(\epsilon)$ are equal to 
to that of $A$, \textit{with the exception} of columns $j_1$, $j'_1$ and rows $i_{2K-1}$, , and $i'_{2M-1}$. Those four lines, however, are mutable: as $j_1$, $i_{2K-1}$, and $i'_{2M-1}$ contain a single fractional value, there are mutable, and $j'_1=j_2$ was chosen as it is mutable.  The constraints on the sums of all rows and columns of $E_1(\epsilon)$ and $E_2(\epsilon)$ are therefore satisfied. In addition, 
\begin{eqnarray*}
\sigma(E_1(\epsilon))&=&\sigma(A) + \epsilon \sigma{N} + \epsilon \sigma(N') \\
&=& k + \epsilon (1 \times (K-1) + (-1) \times K) + \epsilon ( -1 \times (M-1) + 1 \times M) \\
&=& k - \epsilon + \epsilon = k, \\
\sigma(E_2(\epsilon))&=&\sigma(A) - \epsilon \sigma{N} - \epsilon \sigma(N') \\
&=&k - \epsilon (1 \times (K-1) + (-1) \times K) - \epsilon ( -1 \times (M-1) + 1 \times M) \\
&=&k + \epsilon - \epsilon = k.
\end{eqnarray*}
Therefore, $E_1(\epsilon)$ and $E_2(\epsilon)$ belong to $\mathcal{U}^k(R_{min}^{max},C_{min}^{max}$. Since $A = \frac{1}{2}(E_1(\epsilon)+E_2(\epsilon))$, $A$ is not an extreme point of $\mathcal{U}^k(R_{min}^{max},C_{min}^{max})$.
\end{itemize}

 \end{itemize}
This concludes the proof of lemma \ref{lem:lem7}.
\end{proof}

\textit{The proof of theorem \ref{th:th2} b)}  is then just the consequence of lemma \ref{lem:lem5}, \ref{lem:lem6}, and \ref{lem:lem7}.

\vspace{0.2in}

{\bf Acknowledgment.} 
The work discussed here originated from a visit by P.K. at the Institut de Physique Th\'{e}orique, CEA Saclay, France, during the fall of 2022. He thanks them for their hospitality and financial support.

\section*{References}


\end{document}